\documentclass[10pt]{article}
\usepackage{amssymb}
\usepackage{amsmath}
\usepackage{amsfonts}
\usepackage{amsthm}
\usepackage{float}
\usepackage{caption}
\usepackage[english]{babel}
\usepackage[affil-it]{authblk}

\newtheorem{theorem}{Theorem}
\newtheorem{lemma}{Lemma}
\newtheorem{prop}{Proposition}

\newtheorem{corollary}{Corollary}

\title{Completely regular codes in Johnson and Grassmann graphs with small covering radii}

\author{I. Yu. Mogilnykh %
\thanks{\\ 
\noindent E-mail address: \texttt{ivmog@math.nsc.ru}}}
\affil{Sobolev Institute of Mathematics\\Tomsk State University\\
 Novosibirsk State University}
\begin{document}

\maketitle



%% e.g. \dochead{17th International Conference on Dynamical Processes in Excited States of Solids}

\begin{abstract}
%Completely regular codes in Grassmann graphs include several series from spreads and $q$-ary Steiner triple systems.
Let ${\cal L}$ be a Desarguesian 2-spread in the Grassmann graph $J_q(n,2)$.
We prove that the collection of the $4$-subspaces, which do not contain subspaces from ${\cal L}$ is a completely regular code  
in $J_q(n,4)$. Similarly, we construct a completely regular code in the
 Johnson graph $J(n,6)$ from the Steiner quadruple system of the extended Hamming code. We  obtain several new completely regular codes with covering radius $1$ in
the Grassmann graph $J_2(6,3)$ using binary linear programming.
\end{abstract}




\section{Introduction}

The notion of a completely regular code was introduced by Delsarte in \cite{Del} as a generalization of a perfect code. It is known that all perfect codes in  the Grassmann graphs \cite{Chi} are trivial and their nonexistence 
in Johnson graphs is proven for a large number of finite cases \cite{Gord}, \cite{Etz}. Therefore the completely regular codes in these graphs are of interest as they are related to designs and some geometrical objects.

During the years, researchers used several different names for the completely regular codes with covering radius $1$.
These objects could be equivalently defined in terms of: equitable $2$-partitions \cite{MogVal}, \cite{Vor},  perfect $2$-colorings \cite{FDF}, \cite{AvgMog}, \cite{Mog}, \cite{KTV},  \cite{Met}, intriguing sets \cite{DeBruSuz} and others.

We refer to Ph.D. thesis of Martin \cite{MarP} for an introduction on completely regular codes in Johnson graphs and a survey of Borges, Rifa and Zinoviev \cite{Rifa} on a recent progress in the study of completely regular codes in Hamming and Johnson graphs.

In \cite{Del} Delsarte noted interrelations of the completely regular codes in the Hamming and the Johnson graphs with orthogonal arrays and $t$-designs respectively.
In order to emphasize this connection Martin \cite{Mar} suggested the term "a completely regular design of strength $t$". He studied several well-known classes of designs from the point of view of complete regularity and showed that $(k-1)-(n,k,\lambda)$-designs are completely regular. This also holds for $q$-ary designs, therefore $2$-spreads and $2$-ary Steiner triple system  \cite{Brau2} are completely regular. In \cite{AvgMog} it was shown that if $D$ is a $(k-1)-(n,k,1)$-design, then the code of the $(k+1)$-subsets that does not contain any block of $D$ is completely regular in the Johnson graph $J(n,k+1)$. The integer necessary conditions for the existence of $t$-designs were exploited for showing the nonexistence of constant weight perfect codes \cite{Etz}, \cite{Gord} and completely regular codes with covering radius $1$ \cite{Mog}.

The completely regular designs of strength $0$ in the Hamming and the Johnson graphs were characterized by Meyerowitz in \cite{Mey} and have a rather simple structure. The  completely regular codes
of zero strength  in the strongly regular Grassmann graphs are known as the Cameron-Liebler line classes \cite{CL}. 
Contrary to the ordinary designs \cite{Mey} the complete classification of these objects is done only in some particular cases. Several approaches for studying completely regular codes in Johnson graphs  were applied to Cameron-Liebler line classes in \cite{MogG}.

The completely regular codes  of strength $0$ with covering radius $1$ in Grassmann graph $J_q(n,k)$ for $k\geq 3$ could be considered as a generalization of Cameron-Liebler line classes.
Somewhat trivial examples of these codes are the subspaces containing a point;  the subspaces contained in a hyperplane; the subspaces contained in a hyperplane or containing a point for a non-incident point-hyperplane pair. There are not known any other examples of such codes in $J_q(n,k)$ for $k \geq 3$.  

The completely regular codes in the Johnson graphs having strength $1$  were characterized in the following cases: the minimum distance at least $3$ by Martin in \cite{MarS1} and covering radius $1$ recently by Vorob'ev \cite{Vor}.   De Winter and Metsch in \cite{Met} considered completely regular codes with covering radius 1 and strength $1$ in the Grassmann graphs of diameter $3$. They found two new series of examples of such codes: the $3$-subspaces that do not contain a space of $2$-spread and the code arising from symplectic polar space. The first series is similar to a construction from \cite{AvgMog} for completely regular codes in the Johnson graphs $J(n,4)$ and $J(n,5)$ from Steiner triple and quadruples systems.

%Generally speaking there are not so many results devoted to the study of the completely regular codes in the $q$-analogs of the Johnson graphs are concerned few results have been obtained.

%Inducing? 

%In this paper we present  new series of completely regular codes in Johnson and Grassmann graphs with covering radii $1$ and $2$.
%Namely we show that the $4$-subspaces that do not contain a subspace of a Desarguesian $2$-spread is a completely regular code.
%In the case of Johnson graph the code of $6$-subsets that do not contain a quadruple of Steiner quadruple system of the exteneded Hamming code is shown to be completely regular.  

 In Section $2$ we give basic definitions and review the theory developed by Delsarte and Martin in the q-analog case.
In Section $3$ we consider the completely regular codes  with covering radii $2$ in regular graphs.
We show that these sufficient condition for existence of such codes could be formulated in terms of certain eigenvectors of these graphs.
In section $4$ we discuss a spectral property of the inclusion matrix of $t$-subspaces vs $k$-subspaces. We show that the code of $(k+1)$-subspaces not containing subspaces of $(k-1)-(n,k,1)_q$-design is completely regular in the Grassmann graph of $(k+1)$-subspaces. This generalizes a series of De Winter and Metsch \cite{Met} and \cite{AvgMog}.
In Sections $5$ and $6$ we go further by applying this idea to the most symmetric designs, i.e. when 
a $(k-1)-(n,k,1)_q$-design is the Desarguesian $2$-spread or the Steiner quadruple system of extended Hamming code. In this case the code of $(k+2)$-subspaces (subsets), which do not contain  subspaces (subsets) of $(k-1)-(n,k,1)_q$-design is completely regular in the Grassmann (Johnson) graph. 

The linear programming approach is a popular method for construction and classification of codes.
Binary versions of linear programming earlier showed promising results for settling the nonexistence of completely regular codes \cite{Bam} as well as finding these objects \cite{KMV}.  In Section $7$ we obtain several new completely regular designs of strength $1$ and covering radius $1$  in the Grassmann graph $G_2(6,3)$ by binary integer programming with a prescribed subgroup of its automorphism group. We outline the known results on the completely regular codes with $\rho=1$  in this graph in a parameter table.

\section{Definitions and Basic theory}

\subsection{The eigenspaces of the Johnson and Grassmann graphs}\label{secEs}

In what follows we abbreviate $k$-element subset and $k$-dimensional subspace to  $k$-subset and $k$-subspace respectively. 
The vertices of the Johnson graph $J(n,k)$ are $k$-subsets of the set $\{1,\ldots, n\}$ and the edges are pairs of subsets meeting in a $(k-1)$-subset.
The vertices of the Grassmann graph $J_q(n,k)$ are $k$-subspaces of $F_q^n$ and the edges are pairs of subspaces meeting in a $(k-1)$-subspace. These are well-known series of distance-regular graphs. Below we consider $k\leq n/2$.
We also denote the Johnson graph $J(n,k)$ by $J_1(n,k)$ to emphasize that a result holds for Johnson and Grassmann graphs
simultaneously. We use the notations
 $[^n_k]_q$ for $q$-binomial coefficient and $[^n_k]_1$ for its limit value, i.e. the ordinary binomial coefficient.

A  vector $v$ is called an {\it eigenvector} of a graph $\Gamma$ with eigenvalue $\theta$  if it is an eigenvector of the adjacency matrix of $\Gamma$ with eigenvalue $\theta$. 
The following representation for the eigenspaces of the Johnson and Grassmann graphs could be found in \cite{Nor}, see also \cite[Section 4.2]{Del}. We arrange the eigenvalues of $J_q(n,k)$, $q\geq 1$ in descending order starting from zeroeth and denote them by 
$\theta_{i,q}(n,k)$, $i\in \{0,\ldots, k\}$. %Given a code (subset of vertices) $C$ in any of these graphs let $\chi_C$ be the characteristic function of the vertices of $C$ in the vertex set of the graph $J_q(n,k)$. 
 For $i$-subspace ($i$-subset if $q=1$) $X$  let us  consider the characteristic vectors of all $k$-subspaces of $F_q^n$ ($k$-subsets) containing $X$. Let $T_i(k)$ be the linear span of these vectors where $X$ runs through all $i$-subspaces ($i$-subsets) of $F_q^n$ (the set $\{1,\ldots,n\}$).  
\begin{theorem} \cite[Theorem 4.16]{Nor} \label{TheoE}
For any $i, 0\leq i\leq k$, $U_i=T_i(k)\cap T_{i-1}^{\perp}(k)$ is the eigenspace of $J_q(n,k)$, $q\geq 1$,  with eigenvalue 
$$\theta_{i,q}(n,k)=q^{i+1}[^{n-k-i}_{1}]_q[^{k-i}_{1}]_q-[^i_1]_q.$$ 
\end{theorem}
Below we omit the index $q$ in $\theta_{i,q}(n,k)$ when we speak of the eigenvalues of the Johnson graphs.

\subsection{Completely regular codes and $q$-ary designs}

Let $C$ be a code in a regular graph $\Gamma$. 
A vertex $x$ is in $C_i$ if the minimum of the distances between $x$ and the vertices of $C$ is $i$.
The maximum of these distances is called {\it the covering radius} of $C$ and is denoted by $\rho$. The  {\it distance partition} of the vertices of $\Gamma$ with respect to $C_0=C$ is  $\{C_i:i\in \{0,\ldots,\rho\}\}$.

A code $C$ is called {\it completely regular} \cite{Neu} if there are numbers
$\alpha_0,\ldots,\alpha_{\rho}$,  $\beta_0,\ldots,\beta_{\rho-1}$, $\gamma_1,\ldots,\gamma_{\rho}$
such that any vertex of $C_i$ is adjacent to exactly $\alpha_i$, $\beta_i$ and $\gamma_i$ vertices of $C_{i-1}$, $C_{i}$ and $C_{i+1}$
respectively. Note that  $\alpha_0,\ldots,\alpha_{\rho}$ can be found from the remaining numbers and the valency of the graph. The set  
  $\{\beta_0,\ldots,\beta_{\rho-1}; \gamma_1,\ldots,\gamma_{\rho}\}$ is called the {\it intersection array} of the completely regular code $C$.

For a completely regular code $C$ consider the $(\rho+1)\times(\rho+1)$ {\it matrix} $A$  
such that $A_{i,j}$ equals the number of vertices of $C_j$ adjacent to a fixed vertex of $C_i$.
The eigenvalues of the matrix $A$ are called the {\it eigenvalues of the completely regular code} $C$.

\begin{theorem}\label{TheoL}\cite[Theorem 4.5]{CDZ} (Lloyd's theorem)
The eigenvalues of a completely regular code in a graph $\Gamma$ are eigenvalues of $\Gamma$. 
\end{theorem}
%\begin{proof} (Sketch). Suppose $w$ is an eigenvector of $A$ with an eigenvalue $\theta$.
%Define vector $v$ indexed by vertices of $\Gamma$ as follows: $v_x=w_i$ if a vertex $x$ is in $C_i$.
%The definition of a completely regular code imply that $Mv=\theta v$, where $M$ is the adjacency matrix of the graph $\Gamma$.
%\end{proof}
The eigenvalues of a completely regular code with covering radius $1$ are easy to find, see e.g. \cite[Proposition 1]{MogVal}.

\begin{prop}\label{Prop1} Let $C$ be a completely regular code in a $m$-regular graph $\Gamma$ with $\rho=1$ and intersection array $\{\beta_0;\gamma_1\}$.
Then the size of $C$ is $|V(\Gamma)|\gamma_1/(\gamma_1+\beta_0)$ and the eigenvalues of $C$ are $m$ and $m-\beta_0-\gamma_1$.
\end{prop}

In the rest of the section we consider codes in the Johnson or the Grassmann graphs. A collection $D$ of $k$-subspaces ($k$-subsets when $q=1$) of $F_q^n$ (of $\{1,\ldots,n\}$) is a $t-(n,k,\lambda)_q$-{\it design}, if  
any  $t$-subspace  of $F_q^n$ ($t$-element subset of $\{1,\ldots,n\}$) is contained in exactly $\lambda$ elements of $D$.
In throughout of what follows we consider only designs without repeated blocks. The {\it strength} of $D$ is the maximum $t$ such that $D$ is a $t$-design. When $q\geq 2$, a $1-(n,k,1)_q$-design is called a $k$-{\it spread}. 
It is well-known that $k$-spreads exist if and only if $k$ divides $n$.

Let $\chi_C$ be the characteristic vector of a code $C$ in the graph $J_q(n,k)$.
Consider the decomposition of $\chi_C$ over the eigenspaces $U_0,\ldots, U_k$ of $J_q(n,k)$:
\begin{equation}\label{e1}\chi_C=u_0+u_{i_1}+\ldots+u_{i_s}, \end{equation}
where $u_0\in U_0$ and $u_{i_j}\in U_{i_j}$, $j=1,\ldots,s$. The number $s$ in the decomposition (\ref{e1}) is called the {\it dual degree} of the code $C$ \cite{Del}.

We make use of several results that were stated for the completely regular codes in the Johnson graphs by Delsarte \cite{Del} and Martin \cite{Mar}. The arguments of the proofs for the $q$-ary generalization of these results could be obtained by replacing "subset" with "subspace" and follow from the description of the eigenspaces of $J_q(n,k)$ in Section \ref{secEs}.

\begin{theorem}\label{TheoM}
Let $C$ be a code in $J_q(n,k)$, $q\geq 1$ such that the decomposition (\ref{e1}) holds. Then we have the following: 

1. \cite[Theorem 4.2]{Del} The strength of $C$ as $q$-ary design is equal to $min\{i_j:j \in \{1,\ldots,s\}\}-1$.

2. \cite[Theorem 5.13]{Del} If the minimum distance of $C$ is greater or equal to $2s-1$ then $C$ is completely regular.

3. \cite[Theorem 5.10]{Del} The covering radius of $C$ is not greater than $s$.

4. \cite[Corollary 3.4]{Mar} If $C$ is completely regular then its strength is equal to $$min\{i\geq 0:\theta_{i,q}(n,k)\mbox{ is an eigenvalue of } C\}-1.$$    
\end{theorem}

The following statement for the Johnson graphs could be found in \cite[Corollary 1]{Mog}.

\begin{corollary}\label{integer} 
Let $C$ be a completely regular code $C$ in $J_q(n,k)$, $q\geq 1$ with covering radius $1$ and intersection array $\{\beta_0;\gamma_1\}$.
Then we have the following: 

1. The code  $C$ is of size $[^{n}_{k}]_q\beta_0/(\gamma_1+\beta_0)$ and its eigenvalues are numbers $[^{n-k}_1]_q[^{k}_1]_q$ and $[^{n-k}_1]_q[^{k}_1]_q-\gamma_1-\beta_0$.

2. The strength of $C$ is $t$ where $$\theta_{t+1,q}(n,k)=[^{n-k}_1]_q[^{k}_1]_q-\gamma_1-\beta_0.$$ Moreover, the numbers $[^{n-i}_{k-i}]_q\beta_0/(\gamma_1+\beta_0)$ are integers for any $i\in \{0,\ldots,t\}$.
\end{corollary}
\begin{proof}
The eigenvalues of $C$ and the expression $|C|=[^{n}_{k}]_q\beta_0/(\gamma_1+\beta_0)$ follow from Proposition \ref{Prop1} and
the strength of $C$ is by the fourth Statement of Theorem \ref{TheoM}. 
The integer necessary conditions for $t$-design imply that 
for any $i\in \{0,\ldots,t\}$ the number of the subspaces of $C$ containing an $i$-subspace is $[^{n-i}_{k-i}]_q\beta_0/(\gamma_1+\beta_0)$.

\end{proof}

\begin{corollary}\label{Cor1}

1. \cite[Corollary 3.5]{Mar}, \cite[Corollary 8]{AP} For $q\geq 1$ any $(k-1)-(n,k,\lambda)_q$-design is a completely regular code in $J_q(n,k)$ with eigenvalue $\theta_{k,q}(n,k)$ and covering radius $\rho=1$.

2. Any $3$-spread is a completely regular code in the Grassmann graph $J_q(n,3)$ with $\rho=2$. 

\end{corollary}
\begin{proof}
We follow the considerations from \cite[Corollary 3.5]{Mar}. By Statement 1 of Theorem \ref{TheoM} we see that the dual degrees of a $(k-1)-(n,k,\lambda)_q$-design and a $3$-spread  are $1$ and $s$ respectively,  where $s\leq 2$. 
Since covering radius of $3$-spread is $2$, its dual degree is also $2$ by Statement $3$ of Theorem \ref{TheoM}.  
The result follows from Statement $2$ of Theorem \ref{TheoM}.

\end{proof}

%\begin{example}\label{ExSQS}
%Let $C$ be a $(k-1)-(n,k,1)_q$-design, $q\geq 1$. The code $C$ is a completely regular in $J_q(n,k)$ with $\rho=1$ and eigenvalue $\lambda_{k,q}(n,k)$ by Corollary \ref{Cor1}.1. % Since the subspaces (subsets when $q=1$) of ${\cal D}$ are disjoint, we have $\alpha_0=0$ and $\beta_0=[^n_k]_q$. 
%Let $V$ be any of $[^k_1]_q$ $(k-1)$-subspaces (subsets) of a $k$-subspace ($k$-subset) $U$, $U\notin C$. The subspace $V$ belongs to $\lambda$ subspaces (subsets) of $C$ so we have $\gamma_1=\lambda[^k_1]_q$. By Corollary \ref{integer}.1 we have $[^{n-k}_1][^{k}_1]-(\beta_0+\gamma_1)=\theta_{k,q}(n,k)$. Since  $\theta_{k,q}(n,k)$ is $-[^k_1]_q$, $\beta_0=[^{n-k}_1][^{k}_1]-(\lambda-1)[^k_1]_q$.

%\end{example}

%\begin{example}
%Let ${\cal L}$ be a $(k-1)$-spread  in $F_q^n$, which is a completely regular code in $J_q(n,2)$ with $\rho=1$ by Corollary \ref{Cor1}.1.
%Since the subspaces of the spread ${\cal L}$ are disjoint, we have $\alpha_0=0$ and $\beta_0=[^n_2]_q$. Any of $[^2_1]_q$  1-subspaces of a %2-subspace which is not in ${\cal L}$ belongs to a unique subspace of ${\cal L}$ so we have%
%$\gamma_1=q+1$.
%\end{example}

\section{Auxilary statements}

Let the positions of a vector $u$ be indexed by the vertices of a graph $\Gamma$.
If the set of the pairwise different values of $u$ is $\{a_0,\ldots,a_r\}$, denote by
$C^i$ the set of the vertices of $\Gamma$ such that $u_x=a_i$, $i\in \{0,\ldots, r\}$.
The partition $\{C^0,\ldots, C^r\}$ is called the {\it partition  associated } to the vector $u$.
We see that the eigenvectors taking two or three values are tightly related with the completely regular codes of covering radii $1$ and $2$.

\begin{theorem}\label{Theo1}
1. A code $C$ is completely regular in a $m$-regular graph $\Gamma$ with $\rho=1$ and eigenvalue $\theta$, $\theta\neq m$ if and only if $\{C,\overline{C}\}$ is the partition associated to an eigenvector of $\Gamma$ with eigenvalue $\theta$.

2. Let $\Gamma$ be a  $m$-regular graph $\Gamma$, $\{C^0,C^1,C^2\}$ be the partition associated to an eigenvector $u$ of $\Gamma$.
If there are no edges between $C^0$ and $C^2$ and any vertex of $C^1$ is adjacent to 
 exactly $\beta_1$ vertices of $C^2$ then  
 $C^0$ is a completely regular code in $\Gamma$.%and $\{C,C',C''\}$ is a distance partition w.r.t. $C_0$. % Moreover its intersection array is uniquely determined by $a$, $b$, $c$, $\lambda$ and $m$. 
\end{theorem}
\begin{proof}
1. The result could be found in  \cite[Proposition 3.2]{DeBruSuz} or \cite[Proposition 1]{AvgMog}.
%We provide the sketch of the necessary part as we need it in further investigations.
%Let $C$ is a completely regular code with $\rho=1$ and intersection array $\{\beta_0;\gamma_1\}$. Consider
%the vector $v$ indexed by the vertices of $\Gamma$ such that \begin{equation}\label{eq2} v_x=\beta_0 \mbox{ if }x\in C, v_x=-\gamma_1\mbox{ if }x\in \overline{C}.\end{equation} It is not hard to check that $v$ is an eigenvector of graph $\Gamma$ with eigenvalue $m-\beta_0-\gamma_1$ and $(C,\overline{C})$ is the partition associated with $v$.

2. Let $a_0$, $a_1$, $a_2$ be the values of $u$ on $C^0$, $C^1$, $C^2$ respectively.
 Let a vertex $x$ of $C^0$  be adjacent to $\alpha_0(x)$ vertices of $C^0$. Since there are no edges between the vertices $C^0$ and $C^2$, the vertex $x$ is adjacent to $\beta_0(x)=m-\alpha_0(x)$ vertices of $C^1$.
   Consider the sum of the  values of $u$ on the neighbors of $x$.  Since $u$ is an eigenvector with eigenvalue $\theta$ we have the following:

\begin{equation}\label{Theo1eq0}\theta a_0 =(m-\beta_0(x))a_0+\beta_0(x)a_1. \end{equation}
This implies that $\alpha_0(x)$ and $\beta_0(x)$ do not depend on $x$.
The same argument on $C^1$ and $C^2$ implies that $x\in C^2$ has exactly $\gamma_2$ and $\alpha_2$ neighbors in $C^1$ and $C^2$ respectively. This follows from  $\alpha_2+\gamma_2=m$ and the equation
\begin{equation}\label{Theo1eq2}\theta a_2 =(m-\alpha_2)a_1+\alpha_2a_2. \end{equation}

Let a vertex $x$ of $C^1$ be adjacent to $\gamma_1(x)$, $\alpha_1(x)$ and $\beta_1$ vertices of $C^0$, $C^1$ and $C^2$ respectively. 
Note that $\beta_1$ does not depend on $x$ by the condition of the theorem.
Taking into account that $\alpha_1(x)+\beta_1+\gamma_1(x)=m$, the sum of the  values of $u$ on the neighbors of $x$ is

\begin{equation}\label{Theo1eq1}\theta a_1 =\gamma_1(x) a_0+(m-\gamma_1(x)-\beta_1)a_1+\beta_1a_2,\end{equation}

The above implies that $\alpha_1(x)$ and $\gamma_1(x)$ do not depend on $x$ and $C^0$ is a completely regular code by the definition.

\end{proof}

\section{Inducing map and completely regular codes}

For $q\geq 1$ and $l\geq k$ consider the matrix $I_{l,k}$ whose rows and columns are indexed by the vertices of $J_q(n,l)$ and those of $J_q(n,k)$ respectively, $I_{l,k}(x,y)$ is $1$ if $y$ is contained in $x$ and $0$ otherwise. Kantor in \cite{Kant} proved that $I_{l,k}$ is a full rank matrix. 
Moreover, it is clear that the left multiplication by $I_{l,k}$ maps the  column-vectors of the subspace $T_i(k)$ to those of $T_i(l)$  (see Section 2.1 for their definitions). Therefore, $I_{l,k}$ is a bijection from the eigenspace $U_{i,q}(n,k)$
to $U_{i,q}(n,l)$ for any $i\in \{0,\ldots,k\}$.

\begin{theorem}\label{TheoI}
Let $u$ be an eigenvector of $J_q(n,k)$ with eigenvalue $\theta_{i,q}(n,k)$, $q\geq 1$. Then for any $l\geq k$ the vector $I_{l,k}u$ is an eigenvector of $J_q(n,l)$ with eigenvalue $\theta_{i,q}(n,k)$.
\end{theorem}

We make use of the theorem above for obtaining a series of completely regular designs.

\begin{theorem}\label{Corol2} 
1. \cite[Proposition 4]{AvgMog} Let $D$ be a $(k-1)-(n,k,1)$-design. Then 
$$\{U:U\subset \{1,\ldots,n\}, |U|=k+1, |\{V\in D:V\subset U\}|=0\}$$  is a completely regular code in the Johnson graph $J(n,k+1)$ with $\rho=1$.

2. Let $D$ be a $(k-1)-(n,k,1)_q$-design. Then $$\{U:U<F_q^n, dim(U)=k+1, |\{V\in D:V< U\}|=0\}$$
 is a completely regular code in the Grassmann graph $J_q(n,k+1)$ with $\rho=1$.
\end{theorem}

\begin{proof}

2. We use the approach of work \cite{AvgMog} based on the inclusion matrix.
The code $D$ is completely regular with  covering radius $1$ by Corollary \ref{Cor1}.
By the first statement of Theorem \ref{Theo1} there is an eigenvector $u$ of $J_q(n,k)$ with associated partition $\{D,\overline{D}\}$.
A $(k+1)$-subspace of $F_q^n$ contains either $0$ or $1$ subspaces of $D$. These two facts combined  imply that
 $I_{k+1,k}u$ takes only two values.  Moreover, the associated partition to vector $I_{k+1,k}u$  is the partition into the codes $$\{U:U<F_q^n, dim(U)=k+1, |\{V\in D:V< U\}|=0\}\mbox{  and } $$ $$\{U:U<F_q^n, dim(U)=k+1, |\{V\in D:V< U\}|=1\}.$$
By the first statement of Theorem \ref{TheoI} we see that $I_{k+1,k}u$ is an eigenvector of $J_q(n,k+1)$. By the first statement of Theorem \ref{Theo1} we obtain the required.

\end{proof}
{\bf Remark 1}. A combinatorial proof for the statement above in case when $k$ is $2$ (i.e. $2$-spreads) for the Grassmann graphs could be found in \cite[Lemma 12]{Met}. Apart from $2$-spreads the only known example of $(k-1)-(n,k,1)_q$-design, $q\geq 2$ is the Steiner triple system constructed in \cite{Brau2}. This implies the existence of a completely regular code in $J_2(13,4)$ by Theorem \ref{Corol2}.

In the sections below we proceed further with the idea described in Theorem \ref{Corol2} and show that the Steiner quadruple systems of the extended Hamming code and the Desarguesian $2$-spreads produce completely regular codes in  the Johnson graph $J(n,6)$ and the Grassmann graph $J_q(n,4)$ respectively.

\section{Completely regular code in the Johnson graph $J(n,6)$ from the SQS of the extended Hamming code} 
We use the traditional point-block terms throughout the section. 
Consider the design ${\cal Q}$ whose points are the coordinate positions and blocks are the supports of the codewords of weight $4$ of the extended Hamming code of length $n$.
It is well-known that ${\cal Q}$ is a $3-(n,4,1)$-design, also known as a 
{\it Steiner quadruple system}. 

The codes ${\cal Q}$ and  
 $\{x: x \subset \{1,\ldots,n\}, |x|=5, |\{B \in {\cal Q}, B\subset x\}|=0\}$ are completely regular in $J(n,4)$ and $J(n,5)$ respectively by Corollary \ref{Cor1} and Theorem \ref{Corol2}. 
 
\begin{theorem}\label{TheoSQS}
Let ${\cal Q}$ be the Steiner quadruple system of the extended Hamming code of length $n$.
The code $\{x: x \subset \{1,\ldots,n\}, |x|=6, |\{B \in {\cal Q}, B\subset x\}|=0\}$ is completely regular in $J(n,6)$ with $\rho=2$. \end{theorem}
\begin{proof}
 Let $B$ and $B'$ be two blocks of ${\cal Q}$, $|B\cap B'|=2$. Because the extended Hamming code is linear the symmetric difference $B \Delta B'$ is also a block of ${\cal Q}$. Therefore any $6$-subset of $\{1,\ldots,n\}$ contains $0$, $1$ or $3$ blocks of ${\cal Q}$. We consider the following codes:
$$C^0=\{x: x \subset \{1,\ldots,n\}, |x|=6, |\{B \in {\cal Q}: B\subset x\}|=0\},$$
$$C^1=\{x: x \subset \{1,\ldots,n\}, |x|=6, |\{B \in {\cal Q}: B\subset x\}|=1\},$$
$$C^2=\{x: x \subset \{1,\ldots,n\}, |x|=6, |\{B \in {\cal Q}: B\subset x\}|=3\} $$
 that partition the vertices of $J(n,6)$. 

We are now to show that $C^0$ is a completely regular code and  $\{C^0, C^1, C^2\}$ is the distance partition. 
In view of the second statement of Theorem \ref{Theo1} we prove that the vertices of $C^0$ and $C^2$ are disjoint and a vertex of $C^1$ is adjacent to exactly $6$ vertices of $C^2$. We finish the proof by noting that $\{C^0, C^1, C^2\}$ is the partition associated to an eigenvector.
Note that an alternative proof could be done by solely combinatorial arguments for the existence of the intersection array.

A $6$-subset from $C^2$  contains three blocks of ${\cal Q}$. Moreover, the symmetric difference of any two of these blocks is the third block. We see that the following holds:

\begin{equation}\label{eqSQS}\mbox{ for any } x\in C^2\mbox{ and any } i\in x, \mbox{ there is } B\in {\cal Q},i\notin B, B\subset x. \end{equation}
From (\ref{eqSQS}) we conclude that the vertices of $C^0$ and $C^2$ are nonadjacent in the Johnson graph $J(n,6)$. 
%In particular  and the definition of the adjacency in Johnson graph implies that there are no edges between vertices of $C^0$ and $C^2$.
%We see that any vertex of $J(n,6)$ adjacent to $\{1,\ldots,6\}$ contains exactly one block among these three blocks of ${\cal Q}$. 

Let $y$ be a vertex of $C_1$ and $B$ be a unique block of ${\cal Q}$ such that $B\subset y$. 
Let $x\in C^2$ be a vertex that is adjacent to $y$ in $J(n,6)$, i.e. $x=(y\setminus \{i\}) \cup \{j\}$ for some $i\notin y, j\in y$.

Suppose $i$ is in $B$. Then by property (\ref{eqSQS}) there is a block $B'$ of ${\cal Q}$ that is a subset of $x\setminus \{j\}$.
We see that $B'\subset (x\setminus \{j\}) \subset y$.
Moreover $B'$ is not $B$, because 
 $i\in B$, $i\notin x$ and $B' \subset x$.
We have that distinct blocks $B'$ and $B$  from ${\cal Q}$  are subsets of $y$, which contradicts $y\in C^1$. 

We have that $i\in y\setminus B$.
 Then the following blocks of ${\cal Q}$ are subsets of $x$: $B$, $\{s, t, l, j\}$ and $B \Delta\{s, t, l, j\}$ for some $s, t\in B$, $\{l\}=x\setminus (B\cup \{j\})$. On the other hand, given the $2$-subset $\{s,t\}$ of $B$ and the point $l$ from $y\setminus B$ the point $j$ could be reconstructed. Indeed, the block $\{s,t,l,j\}$ is a unique block in $3-(n,4,1)$-design ${\cal Q}$ containing $\{s,t,l\}$.
Since $\{s, t, l, j\}$ and $B \Delta\{s, t, l, j\}$ are contained in the same $x$ from $C^2$, we conclude that there are exactly $[^4_2]_1\cdot 2/2=6$ neighbors of $y$ in $C^2$.

We finally note that the partition $\{C^0,C^1,C^2\}$ is the partition associated to an eigenvector of $J(n,6)$.
The Steiner quadruple system ${\cal Q}$ is a completely regular code with covering radius $1$ and eigenvalue $\theta_{4}(n,4)$ by Corollary \ref{Cor1}. Then by the first statement of Theorem \ref{Theo1} there is an eigenvector $v$ with eigenvalue $\theta_{4}(n,4)$ such that $\{ {\cal Q},\overline {\cal Q}\}$ is the partition associated to $v$. By Theorem \ref{TheoI} the vector $I_{6,4} v$ is an eigenvector of $J(n,6)$ with eigenvalue $\theta_{4}(n,6)$. The definitions of the inclusion matrix $I_{6,4}$ and the codes $C^0$, $C^1$, $C^2$ imply that
these codes form the partition associated to $I_{6,4} v$. By the second statement of Theorem \ref{Theo1} we conclude that $C^0$ is completely regular.
\end{proof}

{\bf Remark 2.} One might obtain the intersection array $\{60,6;4(n-15),6(n-8)\}$ of the code from Theorem \ref{TheoSQS} by combinatorial arguments or following the proof of Theorem \ref{Theo1} from equations (\ref{Theo1eq0})-(\ref{Theo1eq1}).

%\begin{description}
%	\item[-15, x\in C_0]
%	\item[n-18, x\in C_1] 
% 	\item[3n-24, x\in C_2]

\begin{corollary}
The set of the codewords of weight $6$ of the extended Hamming code of length $16$ is a completely regular code in $J(16,6)$ with $\rho=2$.
\end{corollary}
\begin{proof}
The supports of the codewords of weight $6$ are contained in the code $C=\{x: x \subset \{1,\ldots,n\}, |x|=6, |\{ B \in {\cal Q}, B\subset x\}|=0\}$ because otherwise the minimum distance of the extended Hamming code is $2$.
 The definition of a completely regular code and the double counting of edges between $C_i$ and $C_{i+1}$ imply the following: $$|C|\beta_0=|C_1|\gamma_1,\mbox{ } |C_1|\beta_0=|C_2|\gamma_1, \mbox{ }|C|+|C_1|+|C_2|=[^n_6]_1.$$

The intersection array obtained in  Remark $1$, so for $n=16$ we have $$|C|+|C|60/4+|C|6/48=[^{16}_6]=8008$$
and therefore $|C|$ is $448$. Since there are exactly $448$ codewords of weight $6$ in extended Hamming code of length $16$ we conclude that they coincide with $C$, which is 
a completely regular code in $J(16,6)$ by Theorem \ref{TheoSQS}.

\end{proof}

\section{Completely regular code  in $J_q(n,4)$ from Desarguesian $2$-spread}

Let $F'$ be the subfield of the field $F_{q^n}$ of order $q^2$. 
The elements of the multiplicative group of $F_{q^n}$ are parted into the cosets of that of $F'$. We treat $F_{q^n}$ as the vector space $F_q^n$ and any multiplicative coset of $F'$%(appended by zero) 
corresponds to a $2$-subspace of $F_q^n$.
%Moreover any such subspace is closed under the multiplication to the elements of $F'$. 
The collection of such subspaces is a $2$-spread, which is called {\it Desarguesian}. A subspace is called $F'$-{\it closed} if its vectors (threated as elements of $F_{q^n}$) are closed under the multiplication by the elements of $F'$. In particular, the subspaces of a Desarguesian $2$-spread are $F'$-closed. For a subset of $F_q^n$ the minimal inclusion-wise $F'$-closed subspace that contains the subset is called its $F'$-{\it closure}.

Any $2$-spread ${\cal L}$ and all $3$-subspaces
of $F_q^n$ that do not contain any subspace from ${\cal L}$ are completely regular codes in $J_q(n,2)$ and $J_q(n,3)$ respectively.
 In case when ${\cal L}$ is a Desarguesian spread we will show that the code $\{U:U <F_q^n\mbox{, }dim(U)=4, |\{X \in {\cal L}: X<U\}|=0\}$ is a completely regular code in $J_q(n,4)$.

Consider a $4$-subspace $U$ of $F_q^n$. It can contain $0$, $1$ or at least $2$ subspaces from ${\cal L}$. Suppose $X$ and $X'$ are   
$2$-subspaces from ${\cal L}$ that are contained in $U$. Because $X$ and $X'$ meet only in a zero vector, all vectors of $U$ are linear combinations of the vectors of $V$ and $V'$. Moreover since $X$ and $X'$ are $F'$-closed, so is $U$. In other words, the nonzero vectors of $U$ are parted by nonzero vectors from $q^2+1$  subspaces from ${\cal L}$. We have the following code partition of the vertices of $J_q(n,4)$:
$$C^0=\{U:U <F_q^n\mbox{, }dim(U)=4, |\{V \in {\cal L}: V<U\}|=0\},$$
$$C^1=\{U:U <F_q^n\mbox{, }dim(U)=4, |\{V \in {\cal L}: V<U\}|=1\},$$
$$C^2=\{U:U <F_q^n\mbox{, }dim(U)=4, |\{V \in {\cal L}: V<U\}|=q^2+1\}.$$

We now show some structural properties of these codes. 

\begin{lemma}\label{LemGr1}
Any $3$-subspace of $U$, $U\in C^2$ contains exactly one subspace from ${\cal L}$. In particular, the subspaces from $C^0$ and $C^2$ are
nonadjacent in $J_q(n,4)$.
\end{lemma}
\begin{proof}
Let $W$ be a $3$-subspace of $U$ that does not contain subspaces from ${\cal L}$. Since $F_q<F'$ we see that the $F'$-closure of any of  nonzero vectors of $U$ meets $W$ in exactly $q-1$ nonzero vectors. We see that the $F'$-closure of $W$ has at least $(q^3-1)(q+1)$ vectors, so its  dimension is at least $5$. This contradicts that the subspace $U$ is $F'$-closed and that $W$ is a subspace of $U$. 
%The proof is by double counting of  $$|\{(X,V):X\in {\cal L}, X<V<U, dim(V)=3\}|.$$
%Since $U$ is in $C^2$ there are $(q^2+1)(q+1)=q^3+q^2+q+1$ ways to fix $X$ in $U$ and a $3$-subspace of $U$ containing $X$. 
%On the other hand because subspaces from the spread meet in a zero vector,
% any $3$-subspace $V$ contains only one $2$-subspace from ${\cal L}$ or no such subspaces. The statement follows from the fact that
%we have $[^4_1]_q=q^3+q^2+q+1$ $3$-subspaces of $U$.

\end{proof}

 In view of Lemma \ref{LemGr1} one might consider $C^2$ to be the $F'$-closure of all $3$-subspaces that contain exactly one subspace from ${\cal L}$.
Indeed any such $3$-subspace is spanned by a subspace $X$ (which is $F'$-closed) from ${\cal L}$ and a $1$-subspace and the $F'$-closure of the latter one has dimension $2$.

\begin{lemma}\label{LemGr2}
Any subspace $U$, $U\in C^1$ is adjacent to exactly $q+1$ subspaces from $C^2$ in $J_q(n,4)$.
\end{lemma}

\begin{proof}
Let $X\in {\cal L}$ be that such that $X<U$. Let $V\in C^2$ be adjacent to $U$, i.e. $dim(U\cap V)=3$. 
By Lemma \ref{LemGr1} the subspace $U\cap V$ of $U$ must contain $X$. There are exactly $q+1$ $3$-subspaces of $U$ that contain $X$. Their $F'$-closures are in $C^2$  and we obtain the required.% there are $q+1$ subspaces from $C^2$ adjacent to $U$.
\end{proof}

\begin{theorem}\label{TheoR}
Let ${\cal L}$ be a Desarguesian $2$-spread. The code $\{U:U <F_q^n\mbox{, }dim(U)=4, |\{V \in {\cal L}: V<U\}|=0\}$ is completely regular  
in $J_q(n,4)$.

\end{theorem}
\begin{proof} The vertex set of $J_q(n,4)$ is parted into the codes
$C^0$, $C^1$ and $C^2$. By Lemmas \ref{LemGr1} and \ref{LemGr2} there are no edges between $C^0$ and $C^2$
and any vertex from $C^1$ is adjacent to exactly $6$ subspaces from $C^2$.

Since ${\cal L}$ is a completely regular code with covering radius $1$, we see that $\{{\cal L}_0, {\cal L}_1\}=\{{\cal L}, \overline{\cal L}\}$ is the partition associated to an eigenvector $v$ of $J_q(n,2)$ with eigenvalue $\theta_{2,q}(n,2)$. By Theorem \ref{TheoI} the vector $I_{4,2} v$ is an eigenvector of $J_q(n,4)$ with eigenvalue $\theta_{2,q}(n,4)$. The definition of the inclusion matrix $I_{4,2}$ and the definition of the codes $C^0$, $C^1$, $C^2$ imply that these codes form the partition associated to $I_{6,4} v$. The result follows from Theorem \ref{Theo1}.

\end{proof}

We note that following the proof of Theorem \ref{Theo1}  one can obtain that the code given in Theorem \ref{TheoR} has the intersection array
$$\{[^4_1]_q[^3_1]_q,q+1;  q^5(q+1)[^{n-6}_1]_q,[^4_1]_q[^{n-4}_1]_q q\}.$$

\section{Completely regular codes  with $\rho=1$ in $J_2(6,3)$ }

%We followed the approach of \cite{KMV}and construct completely regular codes in $J_2(6,3)$ by binary linear programming method with additional

Throughout this section by {\it the automorphism group} of a code (subset of the vertices) in a graph we mean the setwise stabilizer of the code in the automorphism group of the graph.
Let $G$ be a subgroup of the automorphism group of a $m$-regular graph $\Gamma$. Let   
$O_1,\ldots,O_r$ be the orbits of the action of $G$ on the vertex set of $\Gamma$. Because $O_1,\ldots,O_r$ are orbits we see that given any $i$, $j\in\{1,\ldots,r\}$ any vertex $x$ of $O_i$ is adjacent to exactly $A_{ij}$ vertices of $O_j$
and $A_{ij}$ does not depend on $x$. Let $A$ be the matrix $\{A_{ij}\}_{i,j\in\{1,\ldots,r\}}$. Suppose the automorphism group of a code $C$ has  a subgroup $G$.
 We consider the characteristic vector $\chi_{C,G}$ of $C$ with positions indexed by the orbits $O_1,\ldots,O_r$, i.e. $(\chi_{C,G})_{i}=1$ if and only if $O_i\subseteq C$.
If ${\bf 1}$ is the all-one vector then ${\bf 1}-\chi_{C,G}$ is the characteristic vector of the complement of $C$.
 The notations above imply that 

\begin{equation}\label{eqCRC} A\chi_{C,G}=(m-\beta_0)\chi_{C,G} + \gamma_1({\bf 1}-\chi_{C,G})\end{equation}
holds if and only if $C$ is a completely regular code in $\Gamma$ with covering radius $1$, intersection array $\{ \beta_0; \gamma_1\}$ and $G$ is a subgroup of its automorphism group.

 One might consider (\ref{eqCRC}) to be a binary linear programming problem
with the binary variable vector $\chi_{C,G}$. We consider the application of the binary linear programming approach for showing the existence of completely regular codes in the Grassmann graph $G_2(6,3)$.

Let $a$ be a primitive element of $F_{2^6}$. We set $\Gamma_{21}$ to be the group generated by the multiplication of the vectors 
of $F_2^6$ (treated as the elements of $F_{2^6}$) by $a^{21}$. The vertices of $J_2(6,3)$ are parted into  $465$ orbits of the group $\Gamma_{21}$. 
%One might consider the equation (\ref{eqCRC}) with the variable vector $\chi_C$ is a binary linear programming problem.

Let $C$ be a completely regular code in the graph $J_2(6,3)$ with intersection array $\{\beta_0,\gamma_1\}$.
Due to Corollary \ref{integer} the eigenvalues of $C$ are  the valency of $J_2(6,3)$ which is $98$ and \begin{equation}\label{ev}98-\beta_0-\gamma_1=\theta_{i,2}(6,3),\end{equation}
where $i-1$ is the strength of $C$ as a design.

Let the eigenvalue $98-\beta_0-\gamma_1$ be $\theta_{2,2}(6,3)$, i.e. $C$ is a $q$-ary $1$-design.  In this case a binary linear programming solver found solutions of system (\ref{eqCRC}) for $8$ different values of $\gamma_1$.  Two of the codes  (with $\gamma_1=9$ and $21$) were previously obtained in \cite{Met}.

\begin{theorem}\label{TheoCRC63}
There are completely regular codes in $J_2(6,3)$ with covering radius $1$, intersection array $\{ 93-\gamma_1; \gamma_1\}$
for any $\gamma_1\in\{12,15,18,24,27,30\}$ such that $\Gamma_{21}$ is a subgroup of their automorphism groups.
\end{theorem}

From (\ref{ev}) we see that the intersection array $\{\beta_0,\gamma_1\}$ of any completely regular code could be found from the strength of the code and  $\gamma_1$.
 We summarize the information on the intersection arrays of the completely regular codes in $J_2(6,3)$ with $\rho=1$  in Table $1$.   By integer conditions in the table we mean the  integer necessary existence conditions for designs imposed by the second statement of Corollary \ref{integer}.

\begin{table}[h!]\caption{\normalsize Completely regular codes  in $J_2(6,3)$ with $\rho=1$. The intersection array $\{\beta_0,\gamma_1\}$ of any completely regular code is obtained from its strength and  $\gamma_1$ using  (\ref{ev})}
\label{TableTR15}
\begin{center}
\noindent\begin{tabular}{|c|c|c|c|c|c|}
  \hline
  {\small Eigen-} &{\small Design}&{\small Integer  }&{\small Nonexis-,}&Existence,&Open cases,\\
  {\small value} &{\small strength}&{\small  conditions}&tence, $\gamma_1$& $\gamma_1$& $\gamma_1$\\

   \hline
  $35$ & $0$ & $\gamma_1$ \mbox{ mod }$7=0$& & $7^H, 14^{HP}$  & $21, 28$    \\
\hline
  $5$ & $1$ & $\gamma_1$ \mbox{ mod }$3=0$& $3^{M''}$ & $9^M, 21^{M'}$,  & $3l$, $l\in \{2\}\cup $     \\
	&&&&   $ 3l, l=4,5,6$ & $\{10,\ldots,15\}$\\
	&&&&   $ 8,9,10^A$ & \\\hline
  $-7$ & $2$ & $\gamma_1$ \mbox{ mod }$21=0$& & $21^B, 42^B$ &   \\
  \hline
 \end{tabular}

 \end{center}
\bigskip
%\medskip

$^H$ subspaces belonging to a  hyperplane;

$^{HP}$ subspaces are in a hyperplane $H$ or contain a vector $v$, where $v\notin H$;

$^A$ exists by Theorem \ref{TheoCRC63};

$^{B}$ correspond to  $2-(6,3,3)_2-$ and $2-(6,3,6)_2$-designs, which  exist by \cite{Brau};

$^M$ totally isotropic subspaces
of a symplectic polarity \cite[Example 6]{Met};

$^{M'}$ $3$-subspaces that do not contain subspaces from a $2$-spread;

\noindent \cite[Example 5]{Met}, see also Theorem \ref{Corol2};

$^{M''}$ nonexistence follows from \cite[Lemma 21]{Met}.
%\begin{tabular} 
%\end {tabular}

 \end{table}

{\bf Remark 3.} Apart from the integer necessary conditions there are other techniques for proving the nonexistence of a completely regular code with $\rho=1$ given a putative intersection array. One of them is a method based on the finding weight distribution of the code \cite[Theorem 1]{AvgMog2} (see also \cite{Kro}). In the case of Johnson and Hamming graphs  some intersection arrays feasible by integer necessary conditions were shown to be infeasible using the weight distribution method  \cite{KTV}, \cite{Kro}. However, in the particular instance of the graph $G_2(6,3)$ this approach is not stronger than the integer necessary conditions. A counting argument in \cite[Lemma 21]{Met} implies the nonexistence of the completely regular codes with strength $1$ and $\gamma_1=3$ in this graph.
  
A very intriguing open problem is solving the existence problem for completely regular code in $G_2(6,3)$ with intersection arrays $\{ 56;21\}$ and $\{64 ;28\}$.  These codes are  designs of strength $0$ and therefore could be viewed as
a generalization of Cameron-Liebler line classes in the non-strongly regular case. 

The only known completely regular codes in $G_2(6,3)$ with $\rho=2$ are the code of the subspaces containing a fixed $2$-space and $3$-spread in $G_2(6,3)$ (Corollary \ref{Cor1}).

% \noindent
% {\bf Acknowledgements.} The authors thank the referees for
% careful reading and remarks that improved the presentation of the paper.

\end{document}